\documentclass[12pt]{amsart}

\numberwithin{equation}{section}

\usepackage{amsmath,amsthm,amsfonts,eucal,}
\usepackage{xypic}
\usepackage{graphicx}
\usepackage{amssymb}

\def\cb{{\mathcal B}}

\def\cf{{\mathcal F}}

\def\ch{{\mathcal H}}

\def\cp{{\mathcal P}}

\def\cs{{\mathcal S}}

\def\ga{{\mathfrak A}} 
\def\gb{{\mathfrak B}}

\def\bc{{\mathbb C}}

\def\bi{{\mathbb I}}
\def\bd{{\mathbb D}}
\def\be{{\mathbb E}}
\def\bl{{\mathbb L}}
\def\bj{{\mathbb J}}

\def\bm{{\mathbb M}}
\def\bn{{\mathbb N}}

\def\bp{{\mathbb P}}

\def\bt{{\mathbb T}}

\def\bz{{\mathbb Z}}

\def\a{\alpha}

  \def\G{\Gamma}
\def\d{\delta}

\def\l{\lambda}

\def\p{\pi}

\def\r{\rho}
\def\s{\sigma} 
\def\t{\tau}
\def\f{\varphi}  \def\F{\Phi}
\def\th{\theta}  
\def\om{\omega} \def\Om{\Omega}

\newtheorem{thm}{Theorem}[section]

\newtheorem{cor}[thm]{Corollary}
\newtheorem{prop}[thm]{Proposition}

\theoremstyle{definition}
\newtheorem{rem}[thm]{Remark}
\newtheorem{defin}[thm]{Definition}

\def\min{\mathop{\rm min}}

\def\carf{\mathop{\rm CAR}}

\begin{document}
\title[Failure of the Ryll-Nardzewski  theorem on the CAR algebra]
{Failure of the Ryll-Nardzewski  theorem on the CAR algebra}
\author{Vitonofrio Crismale}
\address{Vitonofrio Crismale\\
Dipartimento di Matematica\\
Universit\`{a} degli studi di Bari\\
Via E. Orabona, 4, 70125 Bari, Italy}
\email{\texttt{vitonofrio.crismale@uniba.it}}

\author{Stefano Rossi}
\address{Stefano Rossi\\
Dipartimento di Matematica\\
Universit\`{a} degli studi di Bari\\
Via E. Orabona, 4, 70125 Bari, Italy}
\email{\texttt{stefano.rossi@uniba.it}}

\begin{abstract}

Spreadability of a  sequence of random variables is a distributional symmetry that is implemented by
suitable actions of $\bj_\bz$, the unital semigroup
of strictly increasing maps on $\bz$ with cofinite range.
We show that $\bj_\bz$ is left amenable but not right amenable, although it does admit
a right F{\o}lner sequence.
This enables us to prove that on the CAR algebra $\carf(\bz)$ there exist spreadable
states that fail to be exchangeable. Moreover, we also show that on $\carf(\bz)$ there exist stationary
states that fail to be spreadable.

\vskip0.1cm\noindent \\
{\bf Mathematics Subject Classification}:  60G09, 43A07, 46L53, 81V74\\
{\bf Key words}: CAR algebra, amenable semigroups, F{\o}lner sequences, invariant states, de Finetti's theorem
\end{abstract}

\maketitle
\section{Introduction}
In classical probability exchangeable sequences of random variables are completely understood. First,
by virtue of a general version of de Finetti's theorem exchangeability is equivalent to
 conditional independence and identical distribution
with respect to the tail algebra of the sequence itself, see {\it e.g.} \cite{Ka}.
Second, exchangeability is the same as spreadability, which is the content of a well-known
result due to  Ryll-Nardzewski, \cite{R}. Putting these two statements together, one obtains
what is known as the extended de Finetti theorem, which represents
an accomplished characterization of exchangeability.\\
The equivalences established in the extended de Finetti theorem, though, will in general cease to hold in the wider context
of non-commutative probability, where a variety of novel phenomena may occur.
For instance, the $W^*$-formalism adopted by K\"{o}stler in
\cite{K} yields examples of quantum stochastic processes which
are spreadable while not being exchangeable.
However, concrete models from non-commutative probability in the $C^*$-formalism do exist
where the extended de Finetti theorem continues to hold.
Boolean, monotone, and $q$-deformed (with $|q|<1$) processes are all a case in point, \cite{CrFid, CFL, Fbo, CFG, CR}.
In the class of the concrete settings alluded to above, the case of the CAR algebra certainly stands out for its
relevance to quantum physics. As for exchangeability, virtually everything is known. Indeed, exchangeable (or symmetric)
states on the CAR algebra make up a Choquet simplex whose extreme points are product states of a single even state
on  $2$ by $2$ matrices, see \cite{CFCMP}. Moreover, a state on the CAR algebra  is symmetric if and only if
the corresponding stochastic  process is conditionally independent and identically distributed with respect to its tail algebra, see \cite{CrFid}.
Even so, spreadability for states on the CAR algebra has not been understood fully insofar as
a description of all spreadable states is still missing.
The present paper aims in part to bridge this gap.
In particular, in Theorem \ref{main} we prove that on the
CAR algebra there exist spreadable states  that are not exchangeable, and stationary states that are not spreadable.
Now, spreadable states are the invariant states under the action of the unital semigroup $\bj_\bz$ of
strictly increasing maps of $\bz$ to itself whose range is a cofinite set.
In order to prove existence of spreadable states
with prescribed values on suitable elements of the CAR algebra, it comes in useful to
delve further into the properties of the semigroup $\bj_\bz$
in terms of amenability. In particular, in Theorem \ref{amen} we prove that
$\bj_\bz$ is left amenable despite having exponential
growth, which is shown in Proposition \ref{exp}. In addition, $\bj_\bz$ is not right amenable, although
it has a right F{\o}lner sequence, as proved in Proposition
\ref{Folner}. It is ultimately this circumstance that allows us to obtain
a good supply of spreadable states  with prescribed properties
which prevent them from being exchangeable.
In more detail, states of this type can be obtained by averaging on
the right F{\o}lner sequence a carefully chosen
quasi-free state associated with a positive Toeplitz operator, Proposition
\ref{toeplitz}.\\
Going back to $\bj_\bz$, we would like to stress that  its left amenability
is a result which has an interest in its own. For instance, the semigroup
$\bj_\bz$ is loosely related to the Thompson monoid $F^+$, which has very recently been
associated with a new distributional invariance principle in \cite{KKW}.
More precisely,
$\bj_\bz$ contains a semigroup $\bd_\bz$  such that
$\bj_\bz\cong\bz\,_{\eta}\!\!\ltimes\bd_\bz$ (the semidirect product is with respect to a suitable action $\eta$ of $\bz$) and
$\bd_\bz$ is isomorphic with a quotient of $F^+$, whose amenability
is not known.\\
A few words on the organization of the paper are in order. After setting the notation and recalling
the necessary definitions from $C^*$-dynamical systems and quantum stochastic processes in Section
\ref{prel}, we directly move on to deal with $\bj_\bz$ in Section \ref{JZ}.
In Section \ref{car} the focus is then on spreadable and stationary states on the CAR algebra and its subalgebra
$\mathfrak{C}$ generated by the so-called position operators. The techniques we develop to treat
the CAR algebra work for $\mathfrak{C}$ as well. In particular, on
$\mathfrak{C}$, too,  there exist spreadable states that are not exchangeable, Corollary \ref{spredC}.
There is however a big difference between the two $C^*$-algebras.
In stark contrast with $\carf(\bz)$, which has a great many exchangeable states,
its subalgebra $\mathfrak{C}$ has in fact only the vacuum as such state, Proposition
\ref{onlyvac}.

\section{Preliminaries}\label{prel}


If $\ga$ is a unital $C^*$-algebra, we denote by ${\rm End}(\ga)$ the set of all unital $*$-endomorphisms of $\ga$. This is a unital
semigroup with respect to the map composition.
A $C^*$-dynamical system is a triplet $(\ga, S,\Gamma)$,
where $\ga$ is a unital $C^*$-algebra, $S$ a unital semigroup, and
$\Gamma: S\rightarrow {\rm End}(\ga)$ a unital homomorphism,  namely
$\Gamma_{gh}=\Gamma_g\circ\Gamma_h$ for all
$g, h\in S$. If $S=G$ is a group,  for any $C^*$-dynamical system $(\ga, G, \a)$, the endomorphism $\a_g$ is a $*$-automorphism of
$\ga$ for all $g\in G$.\\
As is commonly done in the literature, $\cs(\ga)$  denotes the weakly-$*$ compact convex set of all states (normalized, positive, linear functionals) on $\ga$.
For any given $(\ga, S,\Gamma)$, we can define the convex subset $\cs_S(\ga)\subset\cs(\ga)$ of those states of
$\ga$ which are invariant under the action $\Gamma$ of $S$ as

$$
\cs_S(\ga):=\big\{\f\in\cs(\ga)\mid\f\circ\G_g=\f\,,\,\,g\in S\big\}\,.
$$
This is a weakly-$*$ compact convex set as it is closed in $\cs(\ga)$. \\

%


We denote by $\bp_\bz$ the group of finite permutations of the set $\bz$.
Its elements are  bijective maps of $\bz$ which only moves finitely many integers.  The group operation
is given by the map composition.\\
We denote by $\bl_\bz$ the unital semigroup of all strictly increasing maps of $\bz$ to itself.
For any fixed $h\in\bz$, the $h$-{right hand-side partial shift} is the element $\theta_h$ of $\bl_\bz$ given by
$$
\theta_h(k):=\left\{\begin{array}{ll}
                      k & \text{if}\,\, k<h\,, \\
                      k+1 & \text{if}\,\, k\geq h\,.
                    \end{array}
                    \right.
$$
Analogously, the $h$-{left hand-side partial shift} is the element $\psi_h$ of $\bl_\bz$ given by
$$
\psi_h(k):=\left\{\begin{array}{ll}
                      k & \text{if}\,\, k>h\,, \\
                      k-1 & \text{if}\,\, k\leq h\,.
                    \end{array}
                    \right.
$$
The unital semigroup generated by all (left and right) partial shifts is denoted by
$\bi_ \bz$. Furthermore, we denote by $\bd_\bz\subset\bi_\bz$ and by $\be_\bz\subset\bi_\bz$ the submonoids generated by all
right and left partial shifts, respectively.
Finally, let $\bj_\bz\subset\bl_\bz$ be the unital semigroup of all strictly increasing
maps $f$ of $\bz$ to itself whose range is cofinite, that is $|\bz\setminus f(\bz)|<\infty$, where for any
set $E$, $|E|$ denotes the cardinality of $E$.
We recall that $\bi_\bz\subsetneq \bj_\bz \subsetneq \bl_\bz$, see \cite{CFG, CFG2}.
In addition, as proved in \cite[Proposition 4]{CFG2}, the monoid  $\bj_\bz$ can also be recovered as a semidirect product. More precisely, one has
$$\bj_\bz\cong\bz\,_{\eta}\!\!\ltimes\bd_\bz\cong\bz\,_{\eta}\!\!\ltimes\be_\bz $$ where
 $\eta_l(\cdot)=\t^l \cdot \t^{-l}$ for every $l\in\bz$, and $\tau:\bz\rightarrow\bz$
is the map $\t(i):=i+1$, $i\in\bz$.\\

A  {\it quantum stochastic process} is a quadruple
$\big(\ga,\ch,\{\iota_j\}_{j\in \bz},\xi\big)$, where $\ga$ is a unital $C^{*}$-algebra, known as the sample
algebra of the process, $\ch$ is a Hilbert space with
inner product $\langle\cdot,\cdot \rangle$ linear in the first variable,
the maps $\iota_j$ are $*$-homomorphisms from $\ga$ to $\cb(\ch)$, and
$\xi\in\ch$ is a unit vector which is cyclic for  the von Neumann algebra
$\bigvee_{j\in \bz}\iota_j(\ga)$ generated by all $\iota_j(\ga)$'s.\\
We next gather the definitions of the distributional
symmetries that a stochastic process may enjoy and we are concerned with.
\begin{defin}
A stochastic process $\big(\ga,\ch,\{\iota_j\}_{j\in \bz},\xi\big)$
is said to be
\begin{itemize}
\item[{\bf-}] {\it stationary} if
$$
\langle\iota_{j_1}(a_1)\cdots\iota_{j_n}(a_n)\xi,\xi\rangle
=\langle\iota_{j_{1}+1}(a_1)\cdots\iota_{j_{n}+1}(a_n)\xi,\xi\rangle\,;
$$

\item[{\bf-}] {\it exchangeable} if for any $\sigma\in\bp_{\bz}$,
\begin{equation*}
\langle\iota_{j_1}(a_1)\cdots\iota_{j_n}(a_n)\xi,\xi\rangle
=\langle\iota_{\sigma(j_{1})}(a_1)\cdots\iota_{\sigma(j_{n})}(a_n)\xi,\xi\rangle\,;
\end{equation*}
\item[{\bf-}] {\it spreadable}
if for any $g\in\bl_\bz$,
\begin{equation*}
\langle\iota_{j_1}(a_1)\cdots\iota_{j_n}(a_n)\xi,\xi\rangle
=\langle\iota_{g(j_{1})}(a_1)\cdots\iota_{g(j_{n})}(a_n)\xi,\xi\rangle\,.
\end{equation*}
\end{itemize}
where the equalities hold true for all $n\in\bn$, $j_1, j_2, \ldots, j_n\in\bz$, and
$a_1, a_2, \ldots, a_n\in\ga$.
\end{defin}

A stochastic process $\big(\ga,\ch,\{\iota_j\}_{j\in \bz},\xi\big)$ can equivalently be assigned through a state $\varphi$ on
the free product $C^*$-algebra $\ast_{\bz} \ga$.
We recall that $\ast_{\bz} \ga$ is the unital $C^*$-algebra uniquely determined up to isomorphism by the following universal property:
there are unital monomorphisms $i_j:\ga\rightarrow \ast_\bz \ga$ such that for any unital $C^*$-algebra $\gb$ and unital morphisms $\F_j:\ga\rightarrow \gb$, $j\in\bz$, there exists a unique unital homomorphism
$\F:\ast_\bz \ga\rightarrow \gb$ such that $\F\circ i_j=\F_j$ for all $j\in\bz$. For an extensive account of the theory of free products we refer the reader to \cite{A}.\\
On the one hand,  with any stochastic process $\big(\ga,\ch,\{\iota_j\}_{j\in \bz},\xi\big)$ it is possible to associate a state
$\varphi$ on the free product $\ast_\bz \ga$  by setting
$$\varphi(i_{j_1}(a_1)i_{j_2}(a_2)\cdots i_{j_n}(a_n)):=\langle\iota_{j_1}(a_1)\iota_{j_2}(a_2)\cdots\iota_{j_n}(a_n)\xi, \xi\rangle\,,$$
for every $n\in\bn$,  and integers $j_1\neq j_2\neq \ldots\neq j_n$ and $a_1, a_2, \ldots , a_n\in\ga$.
On the other hand, all states on the free product $\ast_\bz \ga$ arise in this way,
see  \cite{CrFid, CrFid2}. Indeed, starting with a state $\f\in\cs\big(\ast_{\bz}\ga\big)$, the corresponding stochastic process
is recovered through the GNS representation $(\p_\f, \ch_\f, \xi_\f)$ of
$\f$ by defining, for every $j\in\bz$, $\iota_j(a):=\pi_\f(i_j(a))$, $a\in\ga$.\\

\noindent
Note that corresponding to any map $g:\bz\rightarrow\bz$ by universality there is a $*$-endomorphism
$\a_g$ of $\ast_{\bz} \ga$ uniquely determined by
$\a_g(i_j(a))=i_{g(j)}(a)$, for all $j\in\bz, a\in\ga$. One has that $\a_{fg}=\a_f\circ\a_g$ for all
$f, g$ maps of $\bz$ to itself.
This means in particular  that $\bp_\bz$ and $\bl_\bz$  act on
$\ast_{\bz} \ga$. Finally, $\bz$ naturally acts on $\ast_{\bz} \ga$ as well through the
$*$-automorphism $\a_\t$ corresponding to the map $\tau:\bz\rightarrow\bz$ we defined above; the corresponding
invariant states are  denoted by $\cs_\bz(\ast_{\bz}\ga)$.\\
The submonoids $\bi_\bz, \bj_\bz\subset\bl_\bz$ act on $\ast_{\bz}\ga$ as well by restriction.
For the purposes of the present paper, it is important to recall
that the set equalities $\cs_{\bl_\bz}(\ast_{\bz}\ga)=\cs_{\bj_\bz}(\ast_{\bz}\ga)=\cs_{\bi_\bz}(\ast_{\bz}\ga)$
hold, see \cite[Remark 4]{CFG2}.
In addition,  in general one has $\cs_{\bp_\bz}(\ast_{\bz}\ga)\subseteq\cs_{\bl_\bz}(\ast_{\bz}\ga)\subseteq\cs_\bz(\ast_{\bz}\ga)$, see
\cite[Formula (2.5)]{CFG}.\\
In light of the one-to-one correspondence between stochastic processes  $\big(\ga,\ch,\{\iota_j\}_{j\in \bz},\xi\big)$
and states on the free product $C^*$-algebra $\ast_{\bz} \ga$, we have that a process:\\

\noindent
(i) is {\it spreadable} if and only if the corresponding state belongs to  $\cs_{\bl_\bz}(\ast_{\bz}\ga)$ or, which is the same, to
 $\cs_{\bj_\bz}(\ast_{\bz}\ga)$, and the state itself is then said to be spreadable;

\vspace{0.15cm}
\noindent
(ii) is {\it exchangeable} if and only if the corresponding state  belongs to  $\cs_{\bp_\bz}(\ast_{\bz}\ga)$, and the state itself is then said to be symmetric;

\vspace{0.15cm}
\noindent
(iii) is {\it stationary} or shift-invariant if and only if the corresponding state  belongs to $\cs_\bz(\ast_{\bz}\ga)$, and the state itself is then said to be stationary.

\section{On the amenability of the monoid $\bj_\bz$}\label{JZ}

This section is devoted to a thorough study of the properties of amenability of the monoid
$\bj_\bz$. In an effort to keep the exposition as self-contained as
possible, we start by recalling a couple of definitions to do with amenable semigroups.\\
A discrete semigroup $S$ is said to be left (or right) amenable if there exists a state $\varphi$ on
$\ell^\infty(S)$ such that $\varphi(l_s f)=\varphi(f)$ (or $\varphi(r_s f)=\varphi(f)$), for every $s\in S$ and $f\in\ell^\infty(S)$, where $l_s f(t):=f(st)$  (or $r_s f(t):=f(ts)$), for any $t\in S$.
For convenience we recall that the weakly-$*$ compact convex set of all states of $\ell^\infty(S)$
can easily be identified with the set of all normalized positive finitely additive measures on $(S, \cp(S))$, where
$\cp(S)$ is the $\s$-algebra of all subsets of $S$, see {\it e.g.} \cite{Pat}.
Unlike the case of groups, left amenability and right
amenability are not the same notion.\\
A left (right)  F{\o}lner sequence of a countable discrete semigroup $S$ is a sequence $\{F_n: n\in\bn\}$ of  finite subsets of $S$ such that for any $h\in S$ one has
$$\lim_{n\rightarrow\infty}\frac{|F_n\Delta hF_n|}{|F_n|}=0\quad \left(\lim_{n\rightarrow\infty}\frac{|F_n\Delta F_nh|}{|F_n|}=0\right)$$
where $\Delta$ denotes the symmetric difference between sets, {\it i.e.} $A\Delta B:= (A\cup B)\setminus(A\cap B)$.\\
We start by showing that $\bj_\bz$ has a right F{\o}lner sequence.

\begin{prop}\label{Folner}
The semigroup $\bj_{\bz}$ has a right F{\o}lner sequence.
\end{prop}

\begin{proof}
As we recalled in Section \ref{prel}, the semigroup isomorphism $\bj_\bz\cong\bz\,_{\eta}\!\!\ltimes\bd_\bz$ holds, where
\begin{equation}\label{semidirect}
\eta_l(\th_m)=\t^l\th_m\t^{-l}=\theta_{m+l}
\end{equation}
for every $l,m\in\bz$. Thus, any $f\in\bj_\bz$ uniquely determines $s\in\bz$ and $h\in\bd_\bz$ such that $f=h\t^s$. In addition, using the relations $\th_k\th_l=\th_l\th_{k-1}$ when $l<k$, any $h\in\bd_\bz\setminus \{\rm{id}_\bz\}$ can be put in the following form:
$$
h=\th_{h_1}^{p_1}\th_{h_2}^{p_2}\cdots \th_{h_r}^{p_r}\,,
$$
for $r\in\bn$, $h_1<h_2<\cdots <h_r\in \bz$, and $p_1,p_2,\ldots,p_r\in \bn$.\\
The proof is constructive. Indeed, we will show that the sequence
$\{F_n: n\in\bn\}$ with
$$
F_n:=\left\{\theta_{-n}^{h_{-n}}\theta_{-n+1}^{h_{-n+1}}\cdots\theta_0^{h_0}\cdots\theta_{n-1}^{h_{n-1}} \theta_n^{h_n}\tau^l: \,\, \sum_{i=-n}^n h_i\leq n^2,\,\, -n\leq l\leq n\right\}
$$
will do.
To this aim, we start by computing the cardinality of each $F_n$. By \cite[p. 161]{GKP},  we have
\begin{equation}\label{cardfn}
|F_n|
= (2n+1)\sum_{k=0}^{n^2}\binom{2n+1+k-1}{k}
=(2n+1)\binom{n^2+2n+1}{n^2}
\end{equation}
Let now $f$ be a fixed element in $\bj_\bz$. If $f=\theta_{i_1}^{k_1}\theta_{i_2}^{k_2}\cdots\theta_{i_r}^{k_r}\tau^s$, with $i_1<i_2<\cdots< i_r$ and
$k_j\in\bn$ for $j=1, 2, \ldots, r$, let us denote by
$M$ the maximum of the finite set $\{|i_1|, |i_2|, \ldots, |i_r|\}$.
Thanks to \eqref{semidirect}, the product of a generic element of $F_n$ with $f$ on the right takes  the form:
\begin{align*}
&\theta_{-n}^{h_{-n}}\theta_{-n+1}^{h_{-n+1}}\cdots\theta_0^{h_0}\cdots\theta_{n-1}^{h_{n-1}} \theta_n^{h_n}\tau^l\theta_{i_1}^{k_1}\theta_{i_2}^{k_2}\cdots\theta_{i_r}^{k_r}\tau^s\\
=&\theta_{-n}^{h_{-n}}\theta_{-n+1}^{h_{-n+1}}\cdots\theta_0^{h_0}\cdots\theta_{n-1}^{h_{n-1}} \theta_n^{h_n}\theta_{i_1+l}^{k_1}\theta_{i_2+l}^{k_2}\cdots\theta_{i_r+l}^{k_r}   \tau^{l+s}\, .
\end{align*}
Now if $|l|\leq n-N$ (for $n$ big enough) where $N:=\max\{|s|, M\}$ ($N$ does not depend on $n$), the above word is  seen to be still in $F_n$ provided that $q+u\leq n^2$, where
${\displaystyle}q=\sum_{i=-n}^n h_i$ and ${\displaystyle}u:=\sum_{j=1}^r k_j$. In more detail, if we define
$n_0:=\max_{|j|\leq n}\{j: h_j\neq 0\}$, the right-hand side of the above equality can be rewritten as
$$\theta_{-n}^{h_{-n}}\theta_{-n+1}^{h_{-n+1}}\cdots\theta_{n_0}^{h_{n_0}}\theta_{i_1+l}^{k_1}\theta_{i_2+l}^{k_2}\cdots\theta_{i_r+l}^{k_r}   \tau^{l+s}\,.$$
Now if $i_1+l\geq n_0$, the word is seen at once to sit in $F_n$ thanks to the conditions imposed on the indices $l$ and
$q$. If $i_1+l<n_0$, then  by virtue of the commutation rules $\th_k\th_l=\th_l\th_{k-1},\, \textrm{for all}\,\, l<k$, there exists $j\in\{1, 2, \ldots, r\}$ such that
our word rewrites as
$$\theta_{-n}^{h_{-n}}\theta_{-n+1}^{h_{-n+1}}\cdots\theta_{i_1+l}^{k_1}\cdots\theta_{i_j+l}^{k_j}\theta_{n_0-j}^{h_{n_0}}\theta_{i_{j+1}+l}^{k_{j+1}}\cdots  \theta_{i_r+l}^{k_r}   \tau^{l+s}\,$$
where $i_{1}+l<\cdots<i_{j}+l\leq n_0-j\leq i_{j+1}+l<\cdots<i_r+l$. By iterating
this procedure as many times as necessary, one ends up with an ordered word which lies in $F_n$
thanks to the constraints on $l$ and $q$.\\
But then we have the inequality
$$|F_n\cap F_nf|\geq {\displaystyle}(2n-2N +1)\sum_{k=0}^{n^2-u}\binom{2n+k}{k}. $$
Therefore, by \eqref{cardfn} we have:
\begin{equation*}
 \frac{|F_n\cap F_n f|}{|F_n|}\geq \frac{2n-2N+1}{2n+1}\,\,\frac{\sum_{k=0}^{n^2-u}\binom{2n+k}{k}}{\binom{n^2+2n+1}{n^2}}\, .
\end{equation*}
Now, the limit of $\frac{2n-2N+1}{2n+1}$ for $n\rightarrow\infty$ is clearly $1$ as is
the limit of ${\displaystyle}\frac{\sum_{k=0}^{n^2-u}\binom{2n+k}{k}}{\binom{n^2+2n+1}{n^2}}
$.
This can be seen by showing that
$$
\lim_{n\rightarrow\infty}\frac{\sum_{k=n^2-u+1}^{n^2}  \binom{2n+k}{k}}{\binom{n^2+2n+1}{n^2}}  
=0\, .
$$
Since the function mapping $k$ to $\binom{2n+k}{k}$ is increasing, the expression above can be bounded in the following way:
\begin{equation*}
\frac{\sum_{k=n^2-u+1}^{n^2}  \binom{2n+k}{k} } {\binom{n^2+2n+1}{n^2}} 
 \leq u\binom{2n+n^2}{n^2} \frac{(n^2)!(2n+1)!}{(n^2+2n+1)!}= u\frac{2n+1}{(n+1)^2}\,,
\end{equation*}
and the thesis follows.
\end{proof}
\begin{rem}\label{leftFolner}
The semigroup $\bj_{\bz}$ also has a left  F{\o}lner sequence. More precisely, if we set 
$$
G_n:=\left\{\tau^l\theta_{-n}^{h_{-n}}\theta_{-n+1}^{h_{-n+1}}\cdots\theta_0^{h_0}\cdots\theta_{n-1}^{h_{n-1}} \theta_n^{h_n}: \,\, \sum_{i=-n}^n h_i\leq n^2,\,\, -n\leq l\leq n\right\}\,,
$$
it is not too hard to verify that for any $f=\tau^s\theta_{i_1}^{k_1}\theta_{i_2}^{k_2}\cdots\theta_{i_r}^{k_r}  \in\bj_{\bz}$ one has
$$\lim_n \frac{|fG_n\cap G_n|}{|G_n|}=1\,. $$
This can  be seen much in the same way as in the proof above, replacing the $n_0$ appearing in that proof with
$\min_{|j|\leq n}\{j: h_j\neq 0\}$.

\end{rem}

In order to prove the left amenability of $\bj_\bz$, we first need to recall
some facts. First, any left-cancellative semigroup $S$ ({\it i.e.} given $s, t, t'\in S$ such that $st=st'$, then $t=t'$)
which admits a left F{\o}lner sequence is left amenable, as proved by Namioka, see \cite[Corollary 4.3]{Tapioka}.
We will also make use of a notion from  semigroup theory which amounts to a weak form of left cancellativity. This is the so-called Klawe condition, \cite{Klawe}:
a semigroup $S$ satisfies the Klawe condition if for any $f, g, s\in S$  the equality $sf=sg$  implies that
there exists $t\in S$ such that  $ft= gt$.

\begin{thm}\label{amen}
The monoid $\bj_\bz$ is left amenable but not right amenable.
\end{thm}

\begin{proof}
Left amenability follows from Remark \ref{leftFolner} and left cancellativity thanks to the result of Namioka
we recalled above.\\
In order to prove that $\bj_\bz$ fails to be right amenable, we will argue by contradiction.
If $\bj_\bz$  were right amenable, then its opposite semigroup $\bj_\bz^{{\rm op}}$ would be left amenable
(we recall that, as a set, $\bj_\bz^{{\rm op}}$ is just $\bj_\bz$ with the new product
$f\cdot^{{\rm op}}g:=gf$, for any $f,g\in \bj_\bz^{{\rm op}}$ ).
Now $\bj_\bz^{{\rm op}}$ has a left F{\o}lner sequence by Proposition \ref{Folner}.
By applying \cite[Proposition 2.5]{Gray}, we would find that $\bj_\bz^{{\rm op}}$ would satisfy the Klawe
condition, which in this case reads as follows.  For any $f, g, s\in\bj_\bz$ the equality $fs=gs$ implies that
there exists $t\in\bj_\bz$ such that  $tf=tg$, that is $f=g$ by left cancellativity of $\bj_\bz$.
But this does not hold true, as is seen by taking
$f=\theta_j$ and $g=s=\theta_{j-1}$, $j\in\bz$, and recalling that $\theta_k\theta_l=\theta_l\theta_{k-1}$ when
$l<k$.
\end{proof}

The next step to accomplish our analysis of  $\bj_\bz$ is to show  it has exponential growth.
This is a result worth stressing because amenability is very often inferred from
subexponential growth, which is of course only a sufficient condition.
More precisely, Theorem 4.4 in \cite{Gray} shows that left amenability
follows from assuming subexponential growth and the Klawe condition.\\
Going back to $\bj_\bz$, we first need to point out
that owing to the relations $\theta_l=\tau^l \theta_0 \tau^{-l}$, $l\in\bz$,  the monoid $\bj_\bz$ is actually finitely generated, with $\tau, \tau^{-1}, \theta_0$ being its generators. We recall that the growth function $f: \bn\rightarrow\bn$ of a finitely
generated (semi)group at $n$ is just the number of different words of length $n$ in the given generators. Thus, in our case we have $f(n)\leq 3^n$, $n\in\bn$. We next aim to show that, however, $f(n)\geq C a^n$, $n\in\bn$, for some real constants
$C$ and $a$ with $a>1$.

\begin{prop}\label{exp}
The monoid $\bj_\bz$ has exponential growth.
\end{prop}

\begin{proof}
We start by defining the sequence of sets
$$
A_n:=\bigg\{\theta_{-n}^{h_{-n}}\theta_{-n+1}^{h_{-n+1}}\cdots\theta_0^{h_0}\cdots\theta_{n-1}^{h_{n-1}} \theta_n^{h_n}:\,\, h_i\geq 1,\, \sum_{i=-n}^n h_i = 3n+1\bigg\}\,.
$$
If we set $k_i:=h_i-1$, $-n\leq i\leq n$, we have $k_i\geq 0$ and $\sum_{i=-n}^n k_i=n$. Therefore,
the cardinality of each $A_n$ is given by
$$|A_n|=\binom{2n+1+n-1}{n}=\binom{3n}{n}=\frac{(3n)!}{n!(2n)!}$$
Using Stirling's approximation of the factorial, $|A_n|$ is then seen to satisfy the asymptotic relation
$|A_n|= O((\frac{27}{4})^n\frac{1}{\sqrt{n}})$.\\
\noindent
If we rewrite the elements of $A_n$ in terms of the generators $\tau, \tau^{-1}, \theta_0$, after the due simplifications we obtain words of the form
$$\tau^{-n}\theta_0^{h_{-n}}\tau\theta_0^{h_{-n+1}}\tau\cdots \tau\theta_0^{h_0}\tau\theta_0^{h_1}\tau\theta_0^{h_2}\cdots \tau\theta_0^{h_n}\tau^{-n}.$$
Now the length of such a word is $7n$.
Phrased differently, the set of all words of length equal to $7n$ in the generators contains the set
$A_n$. Therefore, when $n$ is big enough, we must have
$$f(7n)\geq |A_n|=  O\left(\left(\frac{27}{4}\right)^n\frac{1}{\sqrt{n}}\right)\geq C 6^n
$$
for some constant $C>0$.
 In particular, we find the inequality $f(n)\geq C (\sqrt[7]{6})^n$ for $n$ sufficiently large.
\end{proof}

\section{Stationary and spreadable states on the CAR algebra}
\label {car}

The {\it Canonical Anticommutation
Relations} (CAR for short) algebra over  $\bz$ is the universal unital
$C^{*}$-algebra $\carf(\bz)$,  with unit $I$, generated by
the set $\{a_j, a^{\dagger}_j: j\in \bz\}$ ({\it i.e.} the Fermi
annihilators and creators respectively),  satistying the relations
\begin{equation}\label{stara}
(a_{j})^{*}=a^{\dagger}_{j}\,,\,\,\{a^{\dagger}_{j},a_{k}\}=\d_{j, k}I\,,\,\,
\{a_{j},a_{k}\}=\{a^{\dagger}_{j},a^{\dagger}_{k}\}=0\,,\,\,j,k\in \bz\,.
\end{equation}
where $\{\cdot, \cdot\}$ is the anticommutator and  $\d_{j, k}$ is the Kronecker
symbol.\\
Note that by definition
\begin{equation*}
\carf(\bz)=\overline{\carf{}_0(\bz)}\,,
\end{equation*}
where
$$
\carf{}_0(\bz):=\bigcup\{\carf(F): F\subset \bz\,\text{finite}\,\}
$$
is the (dense) subalgebra of the {\it localized elements}, and
$\carf(F)$ is the $C^*$-subalgebra
generated by the finite set $\{a_j, a^\dag_j: j\in F\}$.\\
$\carf(\bz)$ is a $\bz_2$-graded algebra. The grading is induced by the parity automorphism $\Theta$ acting on the generators as
$$
\Theta(a_{j})=-a_{j}\,,\,\,\Theta(a^{\dagger}_{j})=-a^{\dagger}_{j}\,,\quad
j\in \bz\,.
$$
Consequently, the CAR algebra decomposes as $\carf(\bz)=\carf(\bz)_{+} \oplus\carf(\bz )_{-}$, where
\begin{align*}
&\carf(\bz)_{+}:=\{a\in\carf(\bz) \ | \ \Theta(a)=a\}\,,\\
&\carf(\bz)_{-}:=\{a\in\carf(\bz) \ | \ \Theta(a)=-a\}\,.
\end{align*}
Elements  in $\carf(\bz)_+$ and in $\carf(\bz)_-$ are called
{\it even} and {\it odd}, respectively.\\
A state $\varphi$ on $\carf(\bz)$ is said to be {\it even}
if $\varphi\circ\Theta=\varphi$, which is the same as $\f\lceil_{\carf(\bz)_{-}}=0$.\\

The $C^*$-algebra $\carf(\bz)$ has a distinguished (faithful) irreducible
representation on the Fermi Fock space $\cf_-(\ell^2(\bz))$. In this representation, for every
$j\in\bz$, the operator $a_j^\dag$ (or $a_j$) acts as
the Fermi creator  (or annihilator) of a particle in the state $e_j$,
where $\{e_j:j\in\bz\}$ is the canonical orthonormal basis of
$\ell^2(\bz)$. For an exhaustive account of the Fermi Fock space the reader is
referred to Chapter 5.2 of \cite{BR2}.
The vector state associated with the Fock vacuum vector $\Om\in\cf_-(\ell^2(\bz))$ ({\it i.e.} the one corresponding to the state with no particles at all) is called the vacuum state.\\

The CAR algebra $\carf(\bz)$ is isomorphic to the
$C^{*}$--infinite
tensor product of $\bm_{2}(\bc)$ with itself:
\begin{equation*}
\label{jkw}
\carf(\bz)\cong\overline{\bigotimes_{\bz}\bm_{2}(\bc)}^{C^*}\,,
\end{equation*}
via a Jordan--Klein--Wigner
transformation (see \cite{T3}, Exercise XIV).
Moreover, in Example 3.2 of \cite{CRZbis} the CAR algebra is also shown to be isomorphic
with the infinite graded tensor product of $(\bm_2(\bc), {\rm ad}(U))$ with itself, where
$\bm_2(\bc)$ is understood as a $\bz_2$-graded $C^*$-algebra with grading being induced by
the adjoint action of the unitary (Pauli) matrix $U:=  \left(
\begin{array}{rr}
1 & 0 \\
0 &{ -1}
\end{array}
\right)$.\\
It is worth recalling that the vacuum state can also be obtained
as an infinite product in the sense of Araki-Moriya, see \cite{AM1}, of a particular
even state on $\bm_2(\bc)$. More precisely, by Theorem 5.3 in \cite{CFCMP} any extreme
symmetric state on $\carf(\bz)$ is of the form $\times_\bz \r_\l$ for some $0\leq\l\leq 1$, where
$\r_\l$ is the state on $\bm_2(\bc)$ given by
$\r_\l(T):={\rm Tr}(TD_\l),\, T\in\bm_2(\bc)$
and  $D_\l:=  \left(
\begin{array}{cc}
\l & 0 \\
0 &{ 1-\l}
\end{array}\right)$. The vacuum state is the one corresponding to
$\l=1$.\\
Finally, the CAR algebra can also be seen as a quotient of the free product
$\ast_\bz \bm_2(\bc)$. Indeed, if we define $A:=  \left(
\begin{array}{rr}
0 & 1 \\
0 &0
\end{array}
\right)$, it is easy to see that  the quotient of $\ast_\bz \bm_2(\bc)$
modulo the relations $\{i_j(A^*), i_k(A)\}=\d_{j, k}I$
and $\{i_j(A),i_k(A)\}=\{i_j(A^*),i_k(A^*)\}=0$, for
all $j, k\in\bz$, is isomorphic with the CAR algebra by \eqref{stara}.\\

Note that $\bz, \bp_\bz, \bj_\bz$ act naturally on $\carf(\bz)$ by displacing the indices of the generators
according to the given map of $\bz$. These actions obviously come from
the action at the level of free product we introduced towards the end
of Section \ref{prel}. Therefore, studying invariant Fermi stochastic
processes is the same as analyzing the invariant states of $\carf(\bz)$
under the corresponding action. Also note that exchangeable states are
automatically spreadable and spreadable states are of course stationary.

\medskip

As we recalled in the introduction, de Finetti's theorem
provides quite a satisfactory description of exchangeable
random variables as those which are
conditionally independent and identically
distributed  w.r.t. the tail algebra. This characterization continues to hold true for the CAR algebra, as shown in  \cite[Thorem 5.4]{CrFid}.
However, the Ryll-Nardzewski theorem \cite{R} that exchangeable sequences are the same as spreadable sequences is no longer
true in the CAR algebra.\\
Our next goal is to show that there  exist shift-invariant states on
$\rm{CAR}(\bz)$ that are not spreadable, and spreadable states that are not symmetric.
To this end, we start by singling out a class of (quasi-free) shift-invariant
states, which we do in the next proposition.\\
In the following $\imath$ will denote the imaginary unit of $\bc$.

\begin{prop}\label{toeplitz}
On $\carf(\bz)$ there exists a stationary state
$\om$ such that
$$
\om(a^\dag_ma_n)=\imath \frac{3C}{\pi^2(m-n)^2}\,,
$$
for all $m, n\in\bz$ with $m>n$ and
some positive constant $C$.
\end{prop}
\begin{proof}

As recalled in \cite{FHM}, it is possible to obtain (gauge invariant quasi-free)
stationary states on the CAR algebra by setting for every $n, m\in\bn$ and
$i_1, \ldots, i_m, j_1, \ldots, j_n\in\bz$
\begin{equation}\label{quasifree}
\varphi(a^\dag_{i_1}\cdots a^\dag_{i_m}a_{j_n}\cdots a_{j_1})=\delta_{m, n}
{\rm det} [Q_{i_k, j_l}]_{k, l=1}^n\,,
\end{equation}
where $0\leq Q\leq I$ is a Toeplitz operator on $\ell^2(\bz)$, and
$Q_{m, n}=\langle Qe_m, e_n\rangle$ are its matrix
elements in the canonical basis of $\ell^2(\bz)$. Therefore, we need to show that
a suitable choice of $Q$ yields a state with the desired properties.\\
To this end, we start by recalling that a Toeplitz operator $Q$ is represented
by a bi-infinite matrix $[Q_{m, n}]_{m, n\in\bz}$ where the entries
$Q_{m, n}$ depends only on $(m-n)=:k$. In other words, the entries of such a matrix
are constant along all $k$-th diagonals, $k\in\bz$. Notice that $k=0$ corresponds
to the leading diagonal. For every $k\in\bz$, we denote by $d_k$ the value taken by the entries of our matrix
on the $k$-th diagonal. We now recall that
the operator corresponding to such a matrix will be bounded if and only
if the Fourier series $\sum_{k\in\bz} d_kz^k$ is a function in $L^\infty(\bt)$, see
{\it e.g.} \cite{BG}.
We next verify that the choice

\begin{equation}\label{toepmat}
d_k:=\left\{\begin{array}{ccc}
                      1 & \,\text{if}\,\, k=0\, \\
                      -\imath\frac{3}{\pi^2k^2} & \text{if}\,\, k>0\\
                  \imath\frac{3}{\pi^2k^2}  &\,\, \text{if}\,\, k<0\, .
                    \end{array}
                    \right.
\end{equation}
produces a positive Toeplitz operator. First, note that the Fourier series $\sum_{k\in\bz} d_kz^k$ certainly defines an essentially
bounded function since its sum is even continuous on $\bt$ by total convergence.
Second, note that $d_0 \geq\sum_{k\neq 0} |d_k|$ as follows from $\sum_{k> 0} \frac{1}{k^2}=\frac{\pi^2}{6}$.
We  next show that this inequality implies that the corresponding operator $Q$ is positive. \\
To this end, for every $n\in\bn$ define
a bounded operator $Q^{(n)}$ whose entries $[Q^{(n)}_{i,j}]$ are the same as those of $Q$ for
$|i|, |j|\leq n$ and  $0$ otherwise.
Each $Q^{(n)}$ is a positive operator in that it is represented by a $(2n+1)$-squared  Toeplitz matrix which is Hermitian, diagonally dominant, and with positive diagonal entries. The conclusion will then be reached if we
ascertain that $Q$ is the limit of the sequence $\{Q^{(n)}: n\in\bn\}$ in the weak operator topology.
Because the sequence is bounded in norm,
with $\|Q^{(n)}\|\leq \|Q\|$ for every $n\in\bn$,
it is enough to check for any fixed $i, j\in\bz$ one has
$\lim_n Q^{(n)}_{i, j}= Q_{i,j}$. But this is certainly true since
$Q^{(n)}_{i, j}=Q_{i, j}$ as soon as $n\geq\max\{|i|, |j|\}$.\\
Finally, in order to satisfy the condition $Q\leq I$, it is enough to replace $Q$
 with $\frac{Q}{\|Q\|}$, hence the thesis holds with $C:=\frac{1}{\|Q\|}$.
\end{proof}

\begin{thm}\label{main}
There holds the chain of strict inclusions
$$\cs_{\bp_\bz}(\carf(\bz))\subsetneq\cs_{\bj_\bz}(\carf(\bz))\subsetneq\cs_\bz(\carf(\bz))$$
\end{thm}
\begin{proof}
We start by observing that any state as in Proposition \ref{toeplitz} provides an example of a stationary
state which by construction  fails to be spreadable.\\
Exhibiting a spreadable state that is not exchangeable requires far more work to do.
To this aim, pick a state $\om$ as in Proposition \ref{toeplitz}, then define a sequence $\{\om_n: n\in\bz\}$ of states by setting
$$\om_n:=\frac{1}{|F_n|}\sum_{h\in F_n}\om\circ\alpha_h\,,$$
where $\{F_n\}_{n\in\bn}$ is the right F{\o}lner sequence of $\bj_\bz$ exhibited in
Proposition \ref{Folner}, and
$\bj_\bz\ni h\mapsto\alpha_h\in{\rm End}(\carf(\bz))$ is its natural action on the
CAR algebra.
By weak-$*$ compactness of $\cs(\carf(\bz))$, the sequence above weakly-$*$ converges (up to taking a subsequence) to some state
$\widetilde{\om}$.\\
First, we prove that $\widetilde{\om}$ is spreadable, that is $\widetilde{\om}\circ\alpha_k=\widetilde{\om}$ for
any $k\in\bj_\bz$. This is seen by means of a standard $\frac{\varepsilon}{3}$-argument, which we nevertheless include in full below. We have:
\begin{align*}
\big| \widetilde{\om}(a)-\widetilde{\om}(\alpha_k(a))\big|\leq &
\big| \widetilde{\om}(a)-\om_n(a)\big|
+\big|\om_n(a)-\om_n(\alpha_k(a))\big|\\
+&\big|\om_n(\alpha_k(a))-\widetilde{\om}(\alpha_k(a)) \big|
\end{align*}
Obviously, it is only the second term of the above sum that needs to be taken care of.
For any fixed $\varepsilon>0$ this can be done as follows:
\begin{align*}
\big|\om_n(a)-\om_n(\alpha_k(a))\big|&= \frac{1}{|F_n|}\left|\sum_{h\in F_n}\om(\alpha_h(a))-\sum_{h\in F_n}\om(\alpha_{hk}(a))\right| \\
&=\frac{1}{|F_n|} \left|\sum_{h\in F_n} \om(\alpha_h(a)) - \sum_{h\in F_n k}\om(\alpha_h(a)) \right|\\
&\leq\frac{1}{|F_n|}\sum_{h\in\, F_n\Delta F_n k}|\om(\alpha_h(a))|\\
&\leq \frac{|F_n\Delta F_nk|}{|F_n|}\|a\|\leq\frac{\varepsilon}{3}
\end{align*}
as soon as $n$ is big enough.\\
We claim that $\widetilde{\om}(a_1a^\dag_2)=-\frac{3\imath C}{\pi^2}$ and
$\widetilde{\om}(a_2a^\dag_1)=\frac{3 \imath C}{\pi^2}$.
From this it easily  follows that $\widetilde{\om}$ cannot be exchangeable, for
the equality
$\widetilde{\om}(a_1a^\dag_2)=\widetilde{\om}(a_2a^\dag_1)$ does not hold.\\
We now move on to prove the claim. We only focus on the first equality as the second can be
got to in the same way.\\
For every fixed $n\in\bn$, we define the subset $S_n\subset F_n$ as the set of all
maps $h:\bz\rightarrow\bz$ in $F_n$ such that $h(2)-h(1)>1$.
We next bound the cardinality of each $S_n$ from above.
Recall that a generic element of $F_n$ has the form $$h=\theta_{-n}^{h_{-n}}\theta_{-n+1}^{h_{-n+1}}\cdots\theta_0^{h_0}\cdots\theta_{n-1}^{h_{n-1}} \theta_n^{h_n}\tau^l$$
with $\sum_{i=-n}^n h_i\leq n^2,\,\, -n\leq l\leq n$.
A moment's reflection shows that for any $l$ with $|l|\leq n$, such an $h$ will sit in $S_n$ if and only if
$h_{2+l}\neq 0$ (with $2+l$ still between $-n$ and $n$). This implies that
\begin{equation}\label{estimate}
|S_n|\leq (2n+1)\sum_{k=0}^{n^2-1}\binom{2n-1 +k}{k}=(2n+1)\binom{n^2+2n-1}{n^2-1}\,.
\end{equation}
But then we have
\begin{align*}
\widetilde{\om}(a_1a^\dag_2)&=\lim_n\frac{1}{|F_n|}\sum_{h\in F_n}\om(\alpha_h(a_1a^\dag_2))=
\lim_n\frac{1}{|F_n|}\sum_{h\in F_n}\om(a_{h(1)}a^\dag_{h(2)})\\
&=\lim_n \frac{1}{|F_n|}\left(\sum_{h\in S_n}\om(a_{h(1)}a^\dag_{h(2)})+\sum_{h\in S_n^c}\om(a_{h(1)}a^\dag_{h(2)} )\right)\\
&=\lim_n\frac{1}{|F_n|}\sum_{h\in S_n}\om(a_{h(1)}a^\dag_{h(2)})+\lim_n \frac{|S_n^c|}{|F_n|} \left(-\frac{3\imath C}{\pi^2}\right)\\
&= -\frac{3\imath C}{\pi^2}\, ,
\end{align*}
where we have used that $|\om(a_{h(1)}a^\dag_{h(2)})|\leq 1$ for every $h\in S_n$ and $\frac{|S_n|}{|F_n|}$ tends to $0$, which we need to verify.
From \eqref{cardfn} and \eqref{estimate} we find
\begin{align*}
\frac{|S_n|}{|F_n|}\leq \frac{\binom{n^2+2n-1}{n^2-1}}{\binom{n^2+2n+1}{n^2}}=\frac{n^2(2n+1)}{(n^2+2n+1)(n^2+2n)}= O\bigg(\frac{1}{n}\bigg)\, .
\end{align*}
\end{proof}

We finally turn our attention to the so-called self-adjoint part of ${\rm CAR}(\bz)$.
This is by definition the unital $C^*$-algebra generated by the position operators, say
$\mathfrak{C}:=C^*(x_j: j\in \bz)$, where $x_j:= a_j+a^\dag_j$ for every $j\in\bz$. As a consequence of \eqref{stara}, one has that the $x_j$'s anticommute with one another and their square is the identity, that is
\begin{equation}\label{ac}
x_jx_k+x_kx_j=0,\,\,\,\text{for all}\,j\neq k\,\,\, \text{and}\,\,\, x_j^2=I,\,\,\,\textrm{for all}\, j \in\bz.
\end{equation}

We are going to prove that there is  a marked difference between  $\carf(\bz)$ and its subalgebra
$\mathfrak{C}$ in that the latter has only one symmetric state, the vacuum state.

\begin{prop}\label{onlyvac}
The vacuum state is the only symmetric state on $\mathfrak{C}$, the self-adjoint part of ${\rm CAR}(\bz)$.
\end{prop}

\begin{proof}
By \eqref{ac} one gets that the linear span of words of type
$I$, $x_i$ for $i$ in $\bz$, and finally $x_{j_1} \cdots x_{j_l}$, with $l\geq 2$ and
$j_1, \ldots, j_l\in\bz$ different from one another, is dense in $\mathfrak{C}$.\\
Let $\om$ be a state on $\mathfrak{C}$ invariant under permutation.
We first show that $\om(x_j)=0$ for any $j\in\bz$. Clearly, it is enough
to prove that $\om(x_1)=0$
since $\om(x_j)=\om (x_1)$ for any $j\in\bz$. Therefore, we have the equality $\om(x_1)=\om(\frac{1}{n}\sum_{j=1}^n x_j)$
for any natural $n$. The conclusion will follow if we show that $\|\frac{1}{n}\sum_{j=1}^n x_j\|$
converges to $0$. This is a matter of easy computations. Indeed, by \eqref{ac} we have
\begin{align*}
\|x_1 +\ldots+x_n\|^2&=\|(x_1+\ldots+x_n)^2\|\\
&=\bigg\|nI+\sum_{i< j} (x_ix_j+x_jx_i)\bigg\|=n\,,
\end{align*}
hence $\big\|\frac{1}{n}\sum_{j=1}^n x_j\big\|=\frac{1}{\sqrt{n}}\rightarrow 0$ for $n\rightarrow\infty$.
Longer words can be handled more easily. Indeed, for
any length $l\geq 2$ and any set $\{j_1, j_2, \ldots, j_l\}\subset\bz$ of indices different from one another, we have
$\om(x_{j_1}x_{j_2}\cdots x_{j_l})=\om(x_{j_2}x_{j_1}\cdots x_{j_l})=-\om(x_{j_1}x_{j_2}\cdots x_{j_l})$, and thus $\om(x_{j_1}x_{j_2}\cdots x_{j_l})=0$. The conclusion then follows by density as the restriction of
the vacuum state to  $\mathfrak{C}$ assumes the same values on the above words.
\end{proof}
We would like to end our discussion by pointing out that on $\mathfrak{C}$ as well there exist stationary states that are not spreadable, and spreadable states that are not the vacuum state.

\begin{cor}\label{spredC}
There holds the chain of strict inclusions
$$\cs_{\bp_\bz}(\mathfrak{C})\subsetneq\cs_{\bj_\bz}(\mathfrak{C})\subsetneq\cs_\bz(\mathfrak{C})$$
\end{cor}
\begin{proof}
A stationary state that is not spreadable  is obtained by restricting to $\mathfrak{C}$  the state $\om$
in Proposition \ref{toeplitz}. Indeed, by \eqref{quasifree} and \eqref{toepmat} one easily sees that
$\om(x_1x_2)=\om(a_1a^\dag_2)+\om(a^\dag_1a_2)=-\frac{6\imath C}{\pi^2}$. On the other hand, we have
$\om(x_1x_3)=\om(a_1a^\dag_3)+\om(a^\dag_1a_3)=-\frac{3}{2}\frac{\imath C}{\pi^2}$, which means
the restriction of $\om$ to $\mathfrak{C}$ is not spreadable.\\
A spreadable state that is not the vacuum state  is obtained by restricting to $\mathfrak{C}$  the state
$\widetilde{\om}$ in the proof of Theorem \ref{main}.
Indeed, $\widetilde{\om}$ does not vanish on $x_1x_2$. More precisely,  we have
$$\widetilde{\om}(x_1x_2)=\widetilde{\om}(a_1a^\dag_2)+\widetilde{\om}(a^\dag_1 a_2)=-\frac{6\imath C}{\pi^2}\, ,$$
since $\widetilde{\om}((a_1 a_2)^\sharp)=\lim_n \frac{1}{|F_n|} \sum_{h\in F_n}\om((a_{h(1)}a_{h(2)})^\sharp)=0$
by \eqref{quasifree}, where $\sharp$ is either $1$ or $\dag$.
\end{proof}

\section*{Acknowledgments}
\noindent
We are grateful to R. D. Gray and M. Kambites for some useful discussions concerning
the semigroup $\bj_\bz$.
Finally, we acknowledge  the support of Italian INDAM-GNAMPA.

\end{document}